\documentclass[12pt]{article}

\usepackage{amsmath,amsthm,amsfonts,amssymb,latexsym,amscd,bbm}
\usepackage{color}

\input{prepictex.tex}
\input{pictex.tex}
\input{postpictex.tex}

\begin{document}

\newtheorem{theorem}{Theorem}[section]
\newtheorem{corollary}[theorem]{Corollary}
\newtheorem{lemma}[theorem]{Lemma}
\newtheorem{proposition}[theorem]{Proposition}
\newtheorem{conjecture}[theorem]{Conjecture}
\newtheorem{commento}[theorem]{Comment}
\newtheorem{definition}[theorem]{Definition}
\newtheorem{problem}[theorem]{Problem}
\newtheorem{remark}[theorem]{Remark}
\newtheorem{remarks}[theorem]{Remarks}
\newtheorem{example}[theorem]{Example}

\newcommand{\Nb}{{\mathbb{N}}}
\newcommand{\Rb}{{\mathbb{R}}}
\newcommand{\Tb}{{\mathbb{T}}}
\newcommand{\Zb}{{\mathbb{Z}}}
\newcommand{\Cb}{{\mathbb{C}}}

\newcommand{\Af}{\mathfrak A}
\newcommand{\Bf}{\mathfrak B}
\newcommand{\Ef}{\mathfrak E}
\newcommand{\Gf}{\mathfrak G}
\newcommand{\Hf}{\mathfrak H}
\newcommand{\Kf}{\mathfrak K}
\newcommand{\Lf}{\mathfrak L}
\newcommand{\Mf}{\mathfrak M}
\newcommand{\Rf}{\mathfrak R}

\newcommand{\x}{\mathfrak x}

\def\A{{\mathcal A}}
\def\B{{\mathcal B}}
\def\C{{\mathcal C}}
\def\D{{\mathcal D}}
\def\E{{\mathcal E}}
\def\F{{\mathcal F}}
\def\G{{\mathcal G}}
\def\H{{\mathcal H}}
\def\J{{\mathcal J}}
\def\K{{\mathcal K}}
\def\LL{{\mathcal L}}
\def\N{{\mathcal N}}
\def\M{{\mathcal M}}
\def\N{{\mathcal N}}
\def\OO{{\mathcal O}}
\def\P{{\mathcal P}}
\def\Q{{\mathcal Q}}
\def\SS{{\mathcal S}}
\def\T{{\mathcal T}}
\def\U{{\mathcal U}}
\def\W{{\mathcal W}}

\def\ext{\operatorname{Ext}}
\def\span{\operatorname{span}}
\def\clsp{\overline{\operatorname{span}}}
\def\Ad{\operatorname{Ad}}
\def\ad{\operatorname{Ad}}
\def\tr{\operatorname{tr}}
\def\id{\operatorname{id}}
\def\en{\operatorname{End}}
\def\aut{\operatorname{Aut}}
\def\out{\operatorname{Out}}
\def\coker{\operatorname{coker}}

\def\la{\langle}
\def\ra{\rangle}
\def\rh{\rightharpoonup}

\title{On Conjugacy of MASAs in Graph $C^*$-Algebras}

\author{Tomohiro Hayashi, 
Jeong Hee Hong\footnote{J. H. Hong was supported by Basc Science Research Program through the National 
Researc Foundation of Korea (NRF) funded by te Ministry of Education, Science and Technology 
(Grant No. 2012R1A1A2039991} and 
Wojciech Szyma{\'n}ski\footnote{W. Szyma\'{n}ski was supported by the Villum Fonden Research Grant 
`Local and global structures of groups and their algebras' (2014--2018).}}

\date{{\small 22 April 2016}}

\maketitle

\renewcommand{\sectionmark}[1]{}

\vspace{7mm}
\begin{abstract}
For a large class of finite graphs $E$, we show that whenever $\alpha$ is a vertex-fixing quasi-free automorphism of the 
corresponding graph $C^*$-algebra $C^*(E)$ such that $\alpha(\D_E)\neq\D_E$, where $\D_E$ is the canonical 
MASA in $C^*(E)$, then $\alpha(\D_E)\neq w\D_E w^*$ for all unitaries $w\in C^*(E)$. That is, the two MASAs 
$\D_E$ and $\alpha(\D_E)$ of $C^*(E)$ are outer but not inner conjugate. Passing to an isomorphic $C^*$-algebra by 
changing the underlying graph makes this result applicable to certain non  quasi-free automorphisms as well. 
\end{abstract}

\vspace{3mm}
\noindent {\bf MSC 2010}: 46L05, 46L40

\vspace{3mm}
\noindent {\bf Keywords}: graph $C^*$-algebra, MASA, automorphism


\section{Introduction}

Maximal abelian subalgebras (MASAs) have played very important role in the study of von Neumann algebras 
from the very beginning, and their theory is quite well developed by now, e.g. see \cite{SinSm}, \cite{PV} and 
references therein. Theory of MASAs of 
$C^*$-algebras is somewhat less advanced, several nice attempts in this direction notwithstanding. Our particular 
interest lies in classification of MASAs in purely infinite simple $C^*$-algebras, and especially in Kirchberg algebras. 
In addition to its intrinsic interest, better understanding of MASAs in Kirchberg algebras could have significant consequences 
for the still very much open classification of automorphisms and group actions on these algebras. In this context, we would like to 
single out the recent work of Barlak and Li, \cite{BaLi}, where a connection between the outstanding UCT problem for 
crossed products and existence of invariant Cartan subalgebras is investigated. 

It was unknown until very recently if two outer conjugate MASAs (that is, two MASAs $\Af$ and $\Bf$ for which there 
exists an automorphism $\sigma$ of the ambient algebra such that $\sigma(\Af)=\Bf$) of a purely infinite simple 
$C^*$-algebra must necessarily 
be inner conjugate as well (that is, if there exists a unitary $w$ such that $w\Af w^*=\Bf$). This question was answered to 
the negative in \cite[Theorem 3.7]{CHS2015}, with explicit counter-examples in the Cuntz algebras $\OO_n$. 

In the present paper, we extend the main result of \cite{CHS2015} to the case of purely infinite simple graph 
$C^*$-algebras $C^*(E)$ corresponding to finite graphs $E$. Namely, we show in Theorem \ref{mainqfree} below that 
every quasi-free automorphism of $C^*(E)$ either leaves the canonical MASA $\D_E$ globally invariant or moves it to 
another MASA of $C^*(E)$ which is not inner conjugate to $\D_E$. Although our Theorem \ref{mainqfree} is stated 
for quasi-free automorphisms only, it is in fact applicable to some other automorphisms as well. This is due to the fact 
that passing from one graph $E$ to another $F$ with the isomorphic algebra $C^*(F)\cong C^*(E)$ will often not preserve 
the property of an automorphism to be quasi-free. We discuss this phenomenon in Section 4. 
To make the present paper self-contained, we recall the necessary background on graph 
$C^*$-algebras and their endomorphisms in the preliminaries. 


\section{Preliminaries}

\subsection{Finite directed graphs and their $C^*$-algebras}

Let $E=(E^0,E^1,r,s)$ be a directed graph, where $E^0$ and $E^1$ are {\em finite} sets of vertices 
and edges, respectively, and $r,s:E^1\to E^0$ are range and source maps, respectively. 
A {\em path} $\mu$ of length $|\mu|=k\geq 1$ is a sequence 
$\mu=(\mu_1,\ldots,\mu_k)$ of $k$ edges $\mu_j$ such that 
$r(\mu_j)=s(\mu_{j+1})$ for $j=1,\ldots, k-1$. We view the vertices as paths of length $0$.
The set of all paths of length $k$ is denoted $E^k$, and $E^*$ denotes the collection of 
all finite paths (including paths of length zero). The range and source maps naturally 
extend from edges $E^1$ to paths $E^k$. A {\em sink} is a vertex $v$ which emits 
no edges, i.e. $s^{-1}(v)=\emptyset$. By a {\em cycle} we mean a path $\mu$ of 
length $|\mu|\geq 1$ such that $s(\mu)=r(\mu)$. 
A cycle $\mu=(\mu_1,\ldots,\mu_k)$ has an exit if there is a $j$ such that $s(\mu_j)$ 
emits at least two distinct edges. Graph $E$ is {\em transitive} if for any two vertices $v,w$ there exists 
a path $\mu\in E^*$ from $v$ to $w$ of non-zero length. Thus a transitive graph does not contain any 
sinks or sources. Given a graph $E$, we will denote by $A=[A(v,w)]_{v,w\in E^0}$ its {\em adjacency matrix}.  
That is, $A$ is a matrix with rows and columns indexed by the vertices of $E$, such that $A(v,w)$ is 
the number of edges with source $v$ and range $w$. 

The $C^*$-algebra $C^*(E)$ corresponding to a graph $E$ 
is by definition, \cite{KPRR} and \cite{KPR},  the universal $C^*$-algebra generated by mutually 
orthogonal projections $P_v$, $v\in E^0$, and partial isometries $S_e$, $e\in E^1$, 
subject to the following two relations: 
\begin{description}
\item{(GA1)} $S_e^*S_e=P_{r(e)}$,  
\item{(GA2)} $P_v=\sum_{s(e)=v}S_e S_e^*$ if $v\in E^0$ emits at least one edge. 
\end{description}
For a path $\mu=(\mu_1,\ldots,\mu_k)$ we denote by $S_\mu=
S_{\mu_1}\cdots S_{\mu_k}$ the corresponding partial isometry in $C^*(E)$.  
 We agree to write $S_v=P_v$ for a  $v\in E^0$.
Each $S_\mu$ is non-zero with the domain projection $P_{r(\mu)}$. 
Then $C^*(E)$ is the closed span of $\{S_\mu S_\nu^*:\mu,\nu\in E^*\}$.  
Note that $S_\mu S_\nu^*$ is non-zero if and only if 
$r(\mu)=r(\nu)$. In that case, $S_\mu S_\nu^*$ is a partial isometry with domain and range projections 
equal to $S_\nu S_\nu^*$ and $S_\mu S_\mu^*$, respectively. 

The range projections $P_\mu=S_\mu S_\mu^*$ of all 
partial isometries $S_\mu$ mutually commute, and the abelian $C^*$-subalgebra of $C^*(E)$ 
generated by all of them is called the diagonal subalgebra and denoted $\D_E$. 
We set $\D^0_E = {\rm span}\{P_v  :  v\in E^0  \}$ and, more generally, 
$\D_E^k= {\rm span}\{P_\mu  :  \mu\in E^k  \}$ for $k\geq 0$. $C^*$-algebra $\D_E$ coincides with the 
norm closure of $\bigcup_{k=0}^\infty\D_E^k$. 
If $E$ does not contain sinks and all cycles have exits then 
$\D_E$ is a MASA (maximal abelian subalgebra) in $C^*(E)$ by \cite[Theorem 5.2]{HPP}. 
Throughout this paper, we make the following

\vspace{2mm}\noindent
{\bf standing assumption:}  all graphs we consider are transitive  
and all cycles in these graphs admit exits. 

\vspace{2mm}
There exists a strongly continuous action $\gamma$ of the circle group $U(1)$ on $C^*(E)$, 
called the {\em gauge action}, such that $\gamma_z(S_e)=zS_e$ and $\gamma_z(P_v)=P_v$ 
for all $e\in E^1$, $v\in E^0$ and $z\in U(1)\subseteq\Cb$. 
The fixed-point algebra $C^*(E)^\gamma$ for the gauge action is an AF-algebra, denoted 
$\F_E$ and called the core AF-subalgebra of $C^*(E)$. $\F_E$ is the closed span of 
$\{S_\mu S_\nu^*:\mu,\nu\in E^*,\;|\mu|=|\nu|\}$. For $k\in\Nb=\{0,1,2,\ldots\}$ 
we denote by $\F_E^k$ the linear span of $\{S_\mu S_\nu^*:\mu,\nu\in E^*,\;|\mu|=|\nu|= k\}$. 
$C^*$-algebra $\F_E$ coincides with the norm closure of $\bigcup_{k=0}^\infty\F_E^k$. 

We consider the usual {\em shift} on $C^*(E)$, \cite{CK}, given by
\begin{equation}\label{shift}
\varphi(x)=\sum_{e\in E^1} S_e x S_e^*, \;\;\; x\in C^*(E).  
\end{equation} 
In general, for finite graphs without sinks and sources, the shift is a unital, completely positive map. However, it 
is an injective $*$-homomorphism when restricted to the relative commutant $(\D_E^0)'\cap C^*(E)$. 

We observe that for each $v\in E^0$ projection $\varphi^k(P_v)$ is minimal in the center of $\F_E^k$. 
The $C^*$-algebra $\F_E^k\varphi^k(P_v)$ is the linear span of partial isometries $S_\mu S_\nu^*$ with 
$|\mu|=|\nu|=k$ and $r(\mu)=r(\nu)=v$. It is isomorphic to the full matrix algebra of size $\sum_{w\in E^0}
A^k(w,v)$. The multiplicity of $\F_E^k\varphi^k(P_v)$ in $\F_E^{k+1}\varphi^{k+1}(P_w)$ is $A(v,w)$, 
so the Bratteli diagram for $\F_E$ is induced from the graph $E$, see \cite{CK}, \cite{KPRR} or \cite{BPRSz}.                                                                  
We also note that the relative commutant 
of $\F_E^k$ in $\F_E^{k+1}$ is isomorphic to $\bigoplus_{v,w\in E^0}M_{A(v,w)}(\Cb)$. 

For an integer $m\in\Zb$, we denote by $C^*(E)^{(m)}$ the spectral subspace of the gauge 
action corresponding to $m$. That is,  
$$ C^*(E)^{(m)}:=\{x\in C^*(E) \mid \gamma_z(x)=z^m x,\,\forall z\in U(1)\}. $$ 
In particular, $C^*(E)^{(0)}=C^*(E)^\gamma$. 
There exist faithful conditional expectations $\Phi_{\F}:C^*(E)\to\F_E$ and $\Phi_{\D}:C^*(E)\to\D_E$ 
such that $\Phi_{\F}(S_\mu S_\nu^*)=0$ for $|\mu|\neq|\nu|$ and $\Phi_{\D}(S_\mu S_\nu^*)=0$ 
for $\mu\neq\nu$. Combining $\Phi_\F$ with a faithful conditional expectation from $\F_E$ onto $\F_E^k$, 
we obtain a faithful conditional expectation $\Phi_{\F}^k:C^*(E)\to\F_E^k$. Furthermore, for each 
$m\in\Nb$ there is a unital, contractive and completely bounded map $\Phi^m:C^*(E)\to C^*(E)^{(m)}$ 
given by 
\begin{equation}
\Phi^m(x) = \int_{z\in U(1)} z^{-m}\gamma_z(x)dx. 
\end{equation}
In particular, $\Phi^0=\Phi_\F$.
We have $\Phi^m(x)=x$ for all $x\in C^*(E)^{(m)}$.  If 
$x\in C^*(E)$ and $\Phi^m(x)=0$ for all $m\in\Zb$ then $x=0$.


\subsection{Endomorphisms determined by unitaries}

Cuntz's classical approach to the study of endomorphisms of $\OO_n$, \cite{Cun}, has recently been 
extended to graph $C^*$-algebras in \cite{CHS2015} and \cite{AJSz}. In this subsection, we recall a few 
most essential definitions and facts about such endomorphisms. 

We denote by $\U_E$ the collection of all those unitaries in $C^*(E)$ which 
commute with all vertex projections $P_v$, $v\in E^0$. That is 
\begin{equation}\label{ue}
\U_E:=\U((\D_E^0)'\cap C^*(E)). 
\end{equation}
If $u\in\U_E$ then $uS_e$, $e\in E^1$, are partial isometries in $C^*(E)$ which together with 
projections $P_v$, $v\in E^0$, satisfy (GA1) and (GA2). Thus, by the universality of $C^*(E)$, 
there exists a unital $*$-homomorphism $\lambda_u:C^*(E)\to C^*(E)$ such 
that\footnote{The reader should be aware that in some papers (e.g. in \cite{Cun}) a 
different convention is used, namely $\lambda_u(S_e)=u^* S_e$.}
\begin{equation}\label{lambda}
\lambda_u(S_e)=u S_e \;\;\; {\rm and}\;\;\;  \lambda_u(P_v)=P_v, \;\;\;  {\rm for}\;\; e\in E^1,\; v\in E^0. 
\end{equation}
The mapping $u\mapsto\lambda_u$ establishes 
a bijective correspondence between $\U_E$ and the semigroup of those unital endomomorphisms 
of $C^*(E)$ which fix all  $P_v$, $v\in E^0$.
As observed in \cite[Proposition 1.1]{CHS2012}, if $u\in\U_E\cap\F_E$ then $\lambda_u$ 
is automatically injective. We say $\lambda_u$ is {\em invertible} if $\lambda_u$ is an automorphism of $C^*(E)$. 
We denote 
\begin{equation}
 \Bf := (\D_E^0)'\cap \F_E^1. 
\end{equation}
That is, $\Bf$ is the linear span of elements $S_e S_f^*$, $e,f\in E^1$, with $s(e)=s(f)$. We note that 
$\Bf$ is contained in the multiplicative domain of $\varphi$ and we have $\D_E^1 \subseteq \Bf 
\subseteq \F_E^1$. If $u\in\U(\Bf)$ then $\lambda_u$ is automatically invertible with inverse 
$\lambda_{u^*}$ and the map 
\begin{equation}\label{quasifree}
\U(\Bf) \ni u\mapsto \lambda_u \in\aut(C^*(E))
\end{equation} 
is a group homomorphism with range inside 
the subgroup of {\em quasi-free automorphisms} of $C^*(E)$, see \cite{Z}. Note that this group is almost never trivial 
and it is non-commutative if graph $E$ contains two edges $e,f\in E^1$ such that $s(e)=s(f)$ and $r(e)=r(f)$. 

The shift $\varphi$ globally preserves $\U_E$, $\F_E$ and $\D_E$. For $k\geq 1$ we denote 
\begin{equation}\label{uk}
u_k := u\varphi(u)\cdots\varphi^{k-1}(u).   
\end{equation}
For each $u\in\U_E$ and all $e\in E^1$ we have $S_e u = \varphi(u) S_e$, and thus 
\begin{equation}\label{uaction}
\lambda_u(S_\mu S_\nu^*)=u_{|\mu|}S_\mu S_\nu^*u_{|\nu|}^* 
\end{equation}
for any two paths $\mu,\nu\in E^*$. 


\section{Quasi-free automorphisms}

In this section, we extend the main result of \cite{CHS2015}, applicable to the Cuntz algebras, to a much wider 
class of graph $C^*$-algebras. 

For the proof of Lemma \ref{keylemma}, below, we recall from Lemma 3.2 and Remark 3.3 in \cite{CHS2015} 
that if $x\in C^*(E)$, $x\geq 0$, and $x\D_E=\D_E x$ then $x\in\D_E$. Also, for $v,w\in E^0$ we set 
$$ _v Q_w := \sum_{s(e)=v,\,r(e)=w} S_e S_e^*. $$
Then $_v Q_w$ is a minimal central projection in $\Bf$, and $\Bf_v Q_w$ is isomorphic to $M_{A(v,w)}(\Cb)$. 
Furthermore, $P_v=\sum_{w\in E^0} {_v Q_w}$ is a central projection in $\Bf$. 

\begin{lemma}\label{keylemma}
Let $u\in\U(\Bf)$ be such that $u\D_E^1u^*\neq\D_E^1$, and let $x\in\F_E$ be 
arbitrary. If $x\lambda_u(\D_E) = \D_Ex$ then $x=0$.  
\end{lemma}
\begin{proof}
Suppose $x\in\F_E$ is such that $||x||=1$ and $x\lambda_u(\D_E) = \D_Ex$.  From this we will derive a 
contradiction. 

Since $u\D_E^1u^*\neq\D_E^1$, there exists a vertex $v\in E^0$ such that $u\D_E^1 u^* P_v \neq \D_E^1 P_v$. 
Thus, since $u\D_E^1P_vu^*=   u\D_E^1 u^* P_v$, we can take a projection $p\in\D_E^1P_v$ satisfying
$$ \delta := \inf\{ ||upu^* -q|| \mid q\in\D_E^1P_v\} >0. $$
Since $\Phi_{\F}^1(q')\in\D_E^1$, for all $q'\in\D_E$ we get 
$$ ||upu^*-q'|| \geq ||\Phi_\F^1(upu^*-q')|| = || upu^* - \Phi_\F^1(q') || \geq \delta. $$
By assumption, for each $k\in\Nb$ there is a $q_k\in\D_E$ such that 
\begin{equation}\label{leq1}
x\lambda_u(\varphi^k(p))=q_kx. 
\end{equation} 
Since $u_k\in\F_E^k$ and $\varphi^k(upu^*)\in\varphi^k(\Bf)=(\F_E^k)'\cap\F_E^{k+1}$, we have
\begin{equation}\label{leq2}
\lambda_u(\varphi^k(p)) = u_k\varphi^k(\lambda_u(p))u_k^* = u_k\varphi^k(upu^*))u_k^* = \varphi^k(upu^*). 
\end{equation}
Identities (\ref{leq1}) and (\ref{leq2}) combined yield  
$$ 0 = x\lambda_u(\varphi^k(p)) - q_kx = x\varphi^k(upu^*) - q_kx. $$
Since $upu^*\in\Bf$,  the sequence $\{\varphi^k(upu^*)\}_{k=1}^\infty$ is central in $\F_E$. 
Therefore we have 
$$ \lim_{k\to\infty}(\varphi^k(upu^*)-q_k)xx^*=0. $$
It follows from the assumption on $x$ that $xx^*\D_E = \D_E xx^*$, and thus we may conclude that  $xx^*\in\D_E$. 

Now, take an arbitrary $\epsilon >0$. For a sufficiently large $m\in\Nb$, we have 
$$ \limsup_{k\to\infty} ||(\varphi^k(upu^*)-q_k)\Phi_\F^m(xx^*)|| \leq \epsilon \;\;\; \text{and} 
\;\;\; ||\Phi_\F^m(xx^*)|| \geq 1-\epsilon. $$
Thus we can find a projection $d\in\D_E^m$ such that 
$$ \limsup_{k\to\infty} ||(\varphi^k(upu^*)-q_k)d|| \leq \frac{\epsilon}{1-\epsilon}. $$
Since graph $E$ is transitive, for a sufficiently large $k\in\Nb$ we can find a path $\mu\in E^k$ such that 
$r(\mu)=v$ and $S_\mu S_\mu^* \leq d$. But now we see that 
$$ 3\epsilon \geq ||(\varphi^k(upu^*)-q_k)d|| \geq ||(\varphi^k(upu^*)-q_k)S_\mu S_\mu^*|| = 
|| upu^*P_v - S_\mu^*q_kS_\mu|| \geq \delta. $$
Since $\epsilon$ can be arbitrarily small, this is the desired contradiction. 
\end{proof}

Now, we are ready to prove our main result. 

\begin{theorem}\label{mainqfree}
Let $u\in\U(\Bf)$ be such that $u\D_E^1u^*\neq\D_E^1$. Then there is no non-zero element 
$x\in C^*(E)$ satisfying $x\lambda_u(\D_E) = \D_Ex$. In particular, there is no unitary $w\in C^*(E)$ such 
that $w\D_Ew^*=\lambda_u(\D_E)$. 
\end{theorem}
\begin{proof}
Let $x\in C^*(E)$ be as in the statement of the theorem. To verify that $x=0$, it suffices to show 
that $\Phi^m(x)=0$ for all $m\in\Zb$. 

We have $S_\mu^*\D_E S_\mu = P_{r(\mu)}\D_E $ for each $\mu\in E^*$. Thus $P_{r(\mu)}x
\lambda_u(\D_E) = P_{r(\mu)}\D_E x = S_\mu^*\D_E S_\mu x$, and hence $S_\mu x\lambda_u(\D_E) = 
\D_E S_\mu x$. Therefore by Lemma \ref{keylemma}, we get 
$$ \Phi_\F(S_\mu x) = 0 \;\;\; \text{for all} \;\;\; \mu\in E^*. $$
Let $m\in\Nb$. For a vertex $v\in E^0$ take a path $\mu\in E^m$ with $r(\mu)=v$. Then 
$0 = \Phi_\F(S_\mu x) = S_\mu\Phi^{-m}(x)$. Thus $P_v\Phi^{-m}(x)=0$, and summing over all $v\in E^0$ 
we see that $\Phi^{-m}(x)=0$ for all $m\in\Nb$. 

Now, taking first adjoints of both sides of the identity $x\lambda_u(\D_E) = \D_Ex$ and then applying 
$\lambda_{u^*}=\lambda_u^{-1}$, we get $\lambda_{u^*}(x^*)\lambda_{u^*}(\D_E) = 
\D_E\lambda_{u^*}(x^*)$. Since $u^*\in\U(\Bf)$ and $u^*\D_E^1u\neq\D_E^1$, 
we get $\Phi^{-m}(\lambda_{u^*}(x^*))=0$ for all $m\in\Nb$ applying the preceding argument. But 
$\Phi^{-m}(\lambda_{u^*}(x^*)) = \lambda_{u^*}(\Phi^{-m}(x^*)) = \lambda_{u^*}(\Phi^{m}(x))$. 
Thus $\Phi^m(x)=0$ for all $m\in\Nb$, and the proof is complete. 
\end{proof}

\begin{corollary}
Let $u,v\in\U(\Bf)$ be such that $u\D_E^1u^*\neq v\D_E^1 v^*$. Then there is no unitary $w\in C^*(E)$ such 
that $w\lambda_u(\D_E)w^*=\lambda_v(\D_E)$. 
\end{corollary}


\section{Changing graphs}

The same graph $C^*$-algebra may often be presented by many different graphs, and the property of being quasi-free 
is usually not preserved when passing from one graph to another. This makes Theorem \ref{mainqfree} applicable 
to a much wider class of automorphisms than quasi-free ones, and even in the case of the 
Cuntz algebras gives a larger class of examples of 
outer conjugate MASAs which are not inner conjugate by comparison with \cite[Theorem 3.7]{CHS2015}. 
The following Example \ref{ex1} illustrates this phenomenon. 

\begin{example}\label{ex1}{\rm 
Consider the following graph:
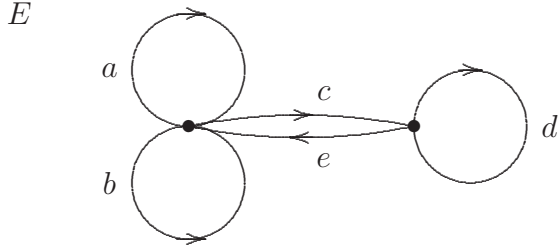
\[ \beginpicture
\setcoordinatesystem units <1.5cm,1.5cm>
\setplotarea x from -2 to 3.5, y from -1 to 1
\put {$\bullet$} at 0 0
\put {$\bullet$} at 2 0 
\put {$E$} at -1.5 1
\put {$a$} at -0.7 0.5
\put {$b$} at -0.7 -0.5
\put {$c$} at 1.2 0.3
\put {$d$} at 3.2 0
\put {$e$} at 1.2 -0.3
\setquadratic
\plot 0 0   1 0.1  2 0 /
\plot 0 0   1 -0.1  2 0 /
\circulararc 360 degrees from 2 0 center at 2.5 0
\circulararc 360 degrees from 0 1 center at 0 0.5
\circulararc 360 degrees from 0 -1 center at 0 -0.5
\arrow <0.25cm> [0.2,0.6] from 0.9 0.1 to 1.1 0.1
\arrow <0.25cm> [0.2,0.6] from 1.1 -0.1 to 0.9 -0.1
\arrow <0.25cm> [0.2,0.6] from -0.1 1 to 0.15 1
\arrow <0.25cm> [0.2,0.6] from -0.1 -1 to 0.15 -1
\arrow <0.25cm> [0.2,0.6] from 2.4 0.5 to 2.6 0.5
\endpicture \] 
Then the graph algebra $C^*(E)$  is isomorphic to the Cuntz algebra $\OO_2=C^*(T_1,T_2)$, \cite{Cun77}, under the identification 
$$ S_a=T_{11}T_1^*, \;\;S_b=T_{121}T_1^*,\;\; S_c=T_{122}T_2^*,\;\; S_d=T_{22}T_2^*,\;\; S_e=T_{21}T_1^*. $$ 
The inverse map is given by 
$$ T_1=S_a + (S_b +S_c)(S_d + S_e)^*,\;\; T_2=S_d + S_e. $$ 
Note that this isomorphism carries $\D_E$ onto the standard diagonal MASA $\D_2$ of $\OO_2$. Let 
$$ \left[ \begin{array}{cc} \xi_{aa} & \xi_{ab} \\ \xi_{ba} & \xi_{bb} \end{array} \right] $$
be a unitary matrix with all entries non-zero complex numbers. Then 
$$ u=\xi_{aa}S_a S_a^* + \xi_{ab}S_a S_b^* + \xi_{ba}S_b S_a^* + \xi_{bb}S_b S_b^* + S_c S_c^* + S_d S_d^* + S_e S_e^* $$ 
is a unitary in $\F_E^1$ satisfying the hypothesis of Theorem \ref{mainqfree}. Consequently, if 
$$ U= \xi_{aa}T_{11}T_{11}^* + \xi_{ab}T_{11}T_{121}^* + \xi_{ba}T_{121}T_{11}^* + \xi_{bb}T_{121}T_{121}^* + 
T_{122}T_{122}^* + T_2T_2^*, $$
then $\lambda_U$ is an automorphism of $\OO_2$ (corresponding to automorphism $\lambda_u$ of $C^*(E)$ via the 
isomorphism $\OO_2\cong C^*(E)$ defined above) such that there is no unitary $w\in\OO_2$ satisfying 
$w\D_2w^*=\lambda_U(\D_2)$. 
} \hfill$\Box$ \end{example}

To get a broader picture, consider an {\em out-splitting} of a graph $E$, as defined in 
\cite{BaPa} (a related construction appears in \cite{Brenken}). Namely, for each $v\in E^0$ 
partition $s^{-1}(v)$ into the union of $m(v)$ non-empty, disjoint subsets $\E_v^1$, $\E_v^2$, \ldots, $\E_v^{m(v)}$. 
Define a new graph $F$, as follows. 
$$ \begin{aligned}
F^0 & = \{v^i \mid v\in E^0, \, 1\leq i \leq m(v)\}, \\
F^1 & = \{e^j \mid e\in E^1,\, 1\leq j\leq m(r(e))\}, 
\end{aligned} $$
with the source and the range maps defined so that for $e\in\E_{s(e)}^i$ we have $s(e^j)=s(e)^i$ and $r(e^j)=r(e)^j$. 
As shown in \cite[Theorem 3.2]{BaPa}, the $C^*$-algebras $C^*(E)$ and $C^*(F)$ are isomorphic by a map that 
carries MASA $\D_E$ onto $\D_F$.  However, 
the groups of quasi-free automorphisms in $C^*(E)$ and $C^*(F)$ may be different. For example, in the following case 
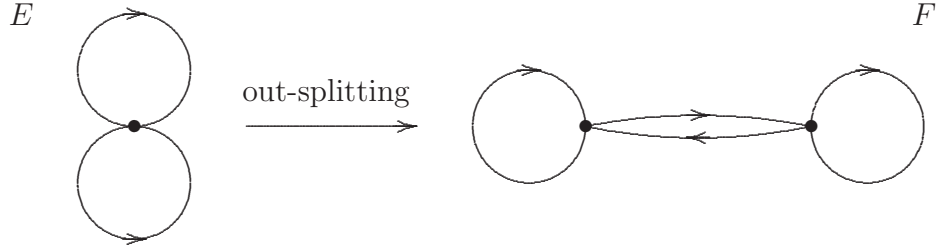
\[ \beginpicture
\setcoordinatesystem units <1.5cm,1.5cm>
\setplotarea x from -6 to 6, y from -1 to 1
\put {$\bullet$} at -4 0
\put {$\bullet$} at 0 0
\put {$\bullet$} at 2 0 
\put {$E$} at -5 1
\put {$F$} at 3 1
\put {out-splitting} at -2.3 0.3
\setlinear
\plot -3 0  -1.5 0 /
\arrow <0.25cm> [0.2,0.6] from -1.7 0 to -1.5 0
\setquadratic
\plot 0 0   1 0.1  2 0 /
\plot 0 0   1 -0.1  2 0 /
\circulararc 360 degrees from 2 0 center at 2.5 0
\circulararc 360 degrees from 0 0 center at -0.5 0
\circulararc 360 degrees from -4 1 center at -4 0.5
\circulararc 360 degrees from -4 -1 center at -4 -0.5
\arrow <0.25cm> [0.2,0.6] from 0.9 0.1 to 1.1 0.1
\arrow <0.25cm> [0.2,0.6] from 1.1 -0.1 to 0.9 -0.1
\arrow <0.25cm> [0.2,0.6] from -4.1 1 to -3.9 1
\arrow <0.25cm> [0.2,0.6] from -4.1 -1 to -3.9 -1
\arrow <0.25cm> [0.2,0.6] from 2.4 0.5 to 2.6 0.5
\arrow <0.25cm> [0.2,0.6] from -0.6 0.5 to -0.4 0.5
\endpicture \] 
the groups $\U(\Bf)$ in $C^*(F)$ and $C^*(E)$ are isomorphic to $U(1)\times U(1)\times U(1)\times U(1)$ 
and $U(2)$, respectively.


\medskip\noindent
Tomohiro Hayashi \\
Nagoya Institute of Technology \\ 
Gokiso-cho, Showa-ku, Nagoya, Aichi \\ 
466--8555, Japan \\
E-mail: hayashi.tomohiro@nitech.ac.jp \\

\smallskip\noindent
Jeong Hee Hong \\
Department of Data Information \\ 
Korea Maritime and Ocean University \\ 
Busan 49112, South Korea \\
E-mail: hongjh@kmou.ac.kr \\

\smallskip\noindent
Wojciech Szyma{\'n}ski\\
Department of Mathematics and Computer Science \\
The University of Southern Denmark \\
Campusvej 55, DK--5230 Odense M, Denmark \\
E-mail: szymanski@imada.sdu.dk

\end{document}